\documentclass[twoside,bezier,12pt, reqno]{amsart}
\usepackage{amssymb,latexsym,epsfig,mathrsfs,graphpap,graphics,amsmath,amssymb}
\usepackage{amstext}
\usepackage{amsthm}
\usepackage[utf8]{inputenc}
\usepackage[english]{babel}
\usepackage{helvet}
\usepackage{courier}

\newtheorem{theorem}{Theorem}[section]
\newtheorem{lemma}{Lemma}[section]

\theoremstyle{Definition}
\newtheorem{definition}{Definition}[section]

\theoremstyle{remark}
\newtheorem{remark}[theorem]{Remark}
\numberwithin{equation}{section}

\begin{document}

\begin{flushleft}
 { {  WIGNER  DISTRIBUTION AND ASSOCIATED UNCERTAINTY PRINCIPLES IN THE FRAMEWORK OF OCTONION LINEAR CANONICAL TRANSFORM  }}

\parindent=0mm \vspace{.2in}

{\bf{ Aamir H. Dar$^{1},$ and M. Younus Bhat$^{2, *}$ }}
\end{flushleft}

{{\it $^{1}$ Department of  Mathematical Sciences,  Islamic University of Science and Technology, Kashmir. E-mail: $\text{ahdkul740@gmail.com}$}}

{{\it $^{2, *}$ Department of  Mathematical Sciences,  Islamic University of Science and Technology, Kashmir. E-mail: $\text{g gyounusg@gmail.com}$}}

\begin{quotation}
\noindent
{\footnotesize {\sc Abstract.} The most recent generalization of octonion Fourier transform (OFT) is the
octonion linear canonical transform (OLCT) that  has become popular in present era due to its applications in  color image and signal processing. On the other hand the applications of Wigner distribution (WD) in signal  and image analysis cannot be excluded. In this paper, we introduce novel integral transform coined as the  Wigner distribution in the octonion linear canonical transform domain (WDOL). We first propose the definition of
the one dimensional WDOL (1D-WDOL), we extend its relationship with 1D-OLCT and 1D-OFT. Then explore  several important properties  of 1D-WDOL, such as reconstruction formula,
Rayleigh's theorem. Second, we introduce the definition of three dimensional WDOL (3D-WDOL) and establish its relationships with the WD associated with quaternion LCT (WD-QLCT) and 3D-WD in LCT domain (3D-WDLCT). Then we study properties like  reconstruction formula,
Rayleigh's theorem and Riemann-Lebesgue Lemma associated with 3D-WDOL.  The crux of this paper lies in developing  well known uncertainty  principles (UPs) including Heisenberg's UP, Logarithmic UP and Hausdorff-Young inequality associated with WDOL\.
 \\


}
\end{quotation}
\section{ \bf Introduction}
\noindent
 Of all the time-frequency distributions, Wigner distribution (WD) \cite{d1}-\cite{d10} is regarded as the most important distribution. WD is considered as an important  frequency analysis tool  that is more suitable for the analysis of time-frequency
characteristics  of chirp-like signals, such as the
linear-frequency-modulated (LFM) signals that are frequently used in wireless communications,
 medical imaging  sonar, radar and many more. For any finite energy signals $f$ and $g$ the WD is defined as \cite{3f}-\cite{4f}
\begin{equation}\label{int 1}
\mathcal W_{f,g}(t,w)=\int_{\mathbb R} f\left(t+\frac{x}{2}\right)g^*\left(t-\frac{x}{2}\right)e^{-iwx}dx,
\end{equation}
where $f\left(t+\frac{x}{2}\right)g^*\left(t-\frac{x}{2}\right)$ represents the instantaneous auto-correlation relation of the
signal $f(x)$. It is also viewed as a local spatial frequency spectrum of the signal, with tremendous  applications in optics, matrix optics, wave optics, geometrical optics, Fourier
 ray optics and radiometry \cite{4z}.\\

In the last few decades, the researcher's community has shown greater  interest in the study of   multidimensional hyper-complex signals defined by means of Cayley–Dickson algebras and there applications in image filtering, watermarking, color and image processing, edge detection and pattern recognition \cite{1b}-\cite{6b}. 
  The Cayley–Dickson algebra of order 4 is known as quaternions. In quaternionic analysis, the quaternion Fourier transform (QFT) is
the most basic and important time-frequency analysis tool for multidimensional quaternionic signals. QFTs are frequently  studied in present era because of its wide  range of applications in  signal and image processing. QFTs are best studied in  \cite{9a}-\cite{12a}. The QFT is regarded as the generalization of the real
and complex FT to the quaternionic case. As in recent times the generalization of integral transforms to quaternion setting is popular, in this regard Yang and Kou \cite{45l} generalized linear canonical transform
(LCT) to quaternion-valued signals, known as the the quaternion LCT (QLCT), that is better effective signal processing tool than the classical  QFT amid its extra parameters, see\cite{13a,14a,15a,16a,17a,18a,WLCT}. Later, El Haoui and Hitzer generalizes the offset linear canonical transform (OLCT) to QOLCT \cite{23l}. Recently Dar and Bhat introduces new integral transform which generalizes the most neoteric quadractic-phase Fourier transform (QPFT) for the quaternion algebra \cite{qqpftl}. Authors in \cite{26l}-\cite{own3} introduced WD associated with QLCT (WD-QLCT) and OLCT (WD-OLCT)and studied their application in detection of quaternion LFM signals. \\
Moving towards other side, Cayley-Dickson algebra of order 8 is known as octonion algebra. Recently Hahn and Snopek  introduces the  octonion Fourier transform (OFT) \cite{hn}. From then OFT is becoming the hot area of research in modern signal processing. Later, authors in \cite{w1,w2,w3,w4} developed and studied the theory of OFT  extensively. Recently Gao and Li  generalized the OFT to the  octonion linear canonical transform (OLCT) \cite{li}  by substituting the Fourier kernel with the LCT kernel. Later Bhat and Dar \cite{ownoolct} introduced the octonion version of offset linear canonical transform.
Furthermore, authors in   \cite{stoft,ownstolct} introduced  octonion short-time Fourier transform and  octonion spectrum of 3d short-time LCT signals where they established classical properties besides establishing Pitt's, Lieb's and uncertainty inequalities. \\

 So motivated and inspired by the merits of Wigner distribution and octonion linear canonical transform, we in this paper propose the novel integral transform coined as the  Wigner distribution in the octonion linear canonical transform domain (WDOL), that gives a unified treatment for already existing classes of
signal processing tools. Therefore it is worthwhile to rigorously study the WDOL and
associated UPs which can be productive for signal processing theory and applications.

\subsection{Paper Contributions}\,\\

The contributions  of this  paper are summarized below:\\

\begin{itemize}
\item To introduce a novel integral transform coined as the  Wigner distribution in the octonion linear canonical transform domain (WDOL)\\

 \item To study the fundamental properties of the 1D-WDOL and 3D-WDOL , including the Reconstrucion Formula, Rayleigh's theorem and Riemann-Lebesgue lemma. \\

  \item To establish the relationship of 3D-WDOL with WD-QLCT and with 3D-WDLCT.\\
      
    \item To formulate several classes of uncertainty principles, such as the Heisenberg UP, logarithmic UP and the Hausdorff-Young inequality associated with WDOL.
\end{itemize}
\subsection{Paper Outlines}\,\\

The  paper is organized as follows: In Section \ref{sec 2}, we gave a brief review to the octonion algebra  and  properties of OFT and OLCT. The definition and the properties of the 1D-WDOL and 3D-WDOL
 are studied in Section \ref{sec 3}. In Section \ref{sec 4}, we establish several classes of uncertainty principles, such as the Heisenberg UP, logarithmic UP and the Hausdorff-Young inequality associated with the proposed transform.. Finally, a conclusion is drawn in Section \ref{sec 5}.

\section{ Preliminaries}
\label{sec 2}
In this section,  we collect some basic facts on the octonion algebra and the octonion Fourier transform (OFT) and octonion linear canonical transform(OLCT), which will be needed throughout the paper.

\subsection{ Octonion algebra} \  \\
The octonion algebra $\mathbb O,$ \cite{23} is generated by the eighth-order Cayley-Dickson construction. In this 
construction, a hyper-complex number  $o\in\mathbb O$ is an ordered pair $q_0,q_1\in\mathbb H$
\begin{eqnarray}\label{o1}
\nonumber o&=&(q_0,q_1)\\
\nonumber&=&((\xi_0,\xi_1),(\xi_2,\xi_3))\\
\nonumber&=&q_0+q_1.\tau_4\\
\nonumber&=&(\xi_0+\xi_1.\tau_2)+(\xi_2+\xi_3.\tau_2).\tau_4\\
\end{eqnarray}
which has equivalent form
\begin{equation}\label{o2}
o=s_o+\sum_{i=1}^7s_i\tau_i=s_0+s_1\tau_1+s_2\tau_2+s_3\tau_3+s_4\tau_4+s_5\tau_5+s_6\tau_6+s_7\tau_7
\end{equation}
thus  $o$ is a hyper-complex number defined by eight real numbers $s_i,i=0,1,\dots,7$ and seven imaginary units $\tau_i$ where $i=1,2,\dots,7.$ It should be noted that the $\mathbb O$ is non-commutative and non-associative algebra. We present a table to represent the multiplication
of imaginary units in the Cayley-Dickson algebra of octonions as .\cite{w2}\\
\begin{center} Table I \end{center}
\begin{center}{\small Multiplication Rules in Octonion Algebra.} \end{center}
\begin{equation*}\label{table}
\begin{array}{|c|c|c|c|c|c|c|c|c|}
\hline
  \cdot & 1 & \tau_1 & \tau_2 & \tau_3 & \tau_4 & \tau_5 & \tau_6 & \tau_7 \\
\hline
  1 & 1 & \tau_1 & \tau_2 & \tau_3 & \tau_4 & \tau_5 & \tau_6 & \tau_7 \\
\hline
  \tau_1 & \tau_1 & -1 & \tau_3 & -\tau_2 & \tau_5 & -\tau_4& -\tau_7& \tau_6 \\
\hline
  \tau_2 & \tau_2 & -\tau_3 & -1 & \tau_1 & \tau_6 & \tau_7 & -\tau_4 & -\tau_5 \\
\hline
  \tau_3 & \tau_3 & \tau_2 & -\tau_1 & -1 & \tau_7 & -\tau_6 & \tau_5 & -\tau_4 \\
\hline
  \tau_4 & \tau_4 & -\tau_5 & -\tau_6 & -\tau_7 & -1 & \tau_1 & \tau_2 & \tau_3 \\
  \hline
  \tau_5 & \tau_5 & \tau_4 & -\tau_7 & \tau_6 & -\tau_1 & -1 & -\tau_3 & \tau_2\\
\hline
  \tau_6 & \tau_6 & \tau_7 & \tau_4 & -\tau_5 & -\tau_2 & \tau_3 & -1 & -\tau_1 \\
  \hline
  \tau_7 & \tau_7 & -\tau_6 & \tau_5& \tau_4 & -\tau_3 & -\tau_2 & \tau_1 & -1 \\
  \hline

\end{array}\end{equation*}
The conjugate of an octonion is defined as
\begin{equation}\label{o3}
\overline{o}=s_0-s_1\tau_1-s_2\tau_2-s_3\tau_3-s_4\tau_4-s_5\tau_5-s_6\tau_6-s_7\tau_7
 \end{equation}
 Therefore norm is defined by $|o|=\sqrt{o\overline{o}}$ and $|o|^2=\sum_{i=o}^7s_i.$ Also $|o_1o_2|=|o_1||o_2|,\forall o_1,o_2\in \mathbb O.$

From (\ref{o1}) it is clear that every $o\in \mathbb O$ can be reshaped in quaternion form like
\begin{equation}\label{o4}
o=a+b\tau_4
\end{equation}
where $a=s_0+s_1\tau_1+s_2\tau_2+s_3\tau_3$ and $b=s_4+s_5\tau_1+s_6\tau_2+s_7\tau_3$ are both quaternions. Evidently, we have  the following lemma.
\begin{lemma}\cite{w2}\label{lem1} Let $a,b\in \mathbb H,$ then\\
(1)\quad $\tau_4a=\overline a\tau_4;$ \qquad\quad(2)\quad$\tau_4(a\tau_4)=-\overline a;$\qquad\quad$(3)\quad(a\tau_4)\tau_4=-a;$\\
(4)\quad$a(b\tau_4)=(ba)\tau_4$;\quad\quad(5)\quad$(a\tau_4)b=(a\overline b)\tau_4;$\quad\quad(6)\quad$(a\tau_4)(b\tau_4)=-\overline b a.$
\end{lemma}
It is evident from the  above result that, for an octonion $a+b\tau_4, a,b \in \mathbb H,$ we have \\
\begin{equation}\label{o5} \overline{a+b\tau_4}=\overline{a}-b\tau_4\end{equation}
and \begin{equation}\label{o6} |a+b\tau_4|^2=|a|^2+|b|^2.\end{equation}
\begin{lemma}\label{lem1}\cite{li}
Let $\tilde o, \hat o\in\mathbb O$. Then $e^{\tilde o}.e^{\hat o}=e^{\tilde o+\hat o}$ iff $\tilde o.\hat o=\hat o.\tilde o$.
\end{lemma}
An octonion-valued function $f:\mathbb R^3\longrightarrow\mathbb O$ has following explicit form
\begin{eqnarray}\label{ofun}
\nonumber f(x)&=&f_0+f_1(x)\tau_1+f_2(x)\tau_2+f_3(x)\tau_3+f_4(x)\tau_4+f_5(x)\tau_5+f_6(x)\tau_6+f_7(x)\tau_7\\
\nonumber&=&f_0+f_1\tau_1+(f_2+f_3\tau_1)\tau_2+[f_4+f_5\tau_1+(f_6+f_7\tau_1)\tau_2]\tau_4\\
&=&\tilde f(x)+\hat f(x)\tau_4
\end{eqnarray}
where each $f_i(x)$ is a real valued functions, $\tilde f,\hat f \in \mathbb H$ are as in(\ref{o1}) and $x=(x_1,x_2,x_3)\in \mathbb R^3.$

\subsection{ Octonion Fourier Transform and Octonion Linear Canonical Transform} \  \\
Let $f:\mathbb R\rightarrow \mathbb O$ be an octonion-valued function, then 1D octonion Fourier transform (OFT)\cite{li,1doft} is given by 
\begin{equation}\label{eqn 1doft}
\mathcal F_{\tau_4}[f](w)=\int_{\mathbb R}f(x)e^{-\tau_4 2\pi x w}dx,
\end{equation}
and its inverse is given by 
\begin{equation}\label{eqn inv 1doft}
f(x)=\mathcal F^{-1}_{\tau_4}\left\{\mathcal F_{\tau_4}[f]\right\}(x)=\int_{\mathbb R}\mathcal F_{\tau_4}[f](w)e^{\tau_4 2\pi x w}dw.
\end{equation} 
See \cite{1doft} for the properties of 1D OFT.\\

Later, authors in \cite{w2,w4} introduced 3D OFT of an octonion-valued signal $f\in L^1(\mathbb R^3,\mathbb O)\cap L^2(\mathbb R^3,\mathbb O)$ as
\begin{equation}\label{eqn 3doft}
\mathcal F_{\tau_1,\tau_2,\tau_4}[f]({\bf w})=\int_{\mathbb R}f(x)e^{-\tau_1 2\pi x_1 w_1}e^{-\tau_2 2\pi x_2 w_2}e^{-\tau_4 2\pi x_3 w_3}d{\bf x},
\end{equation}
where ${\bf x}=(x_1,x_2,x_3)\in \mathbb R^3$,  ${\bf w}=(w_1,w_2,w_3)\in \mathbb R^3$ and  multiplication in above integral is done from left to right. Also the order of imaginary units in (\ref{eqn 3doft}) is fixed because octonions  are neither commutative nor associative.\\

The inverse 3D OFT is given as
\begin{equation}\label{eqn inv  3doft}
f({\bf x})=\mathcal F^{-1}_{\tau_1,\tau_2,\tau_4}\left\{\mathcal F_{\tau_1,\tau_2,\tau_4}[f]\right\}({\bf x})=\int_{\mathbb R}f(x)e^{\tau_4 2\pi x_3 w_3}e^{\tau_2 2\pi x_2 w_2}e^{\tau_1 2\pi x_1 w_1}d{\bf w},
\end{equation}
where ${\bf x}=(x_1,x_2,x_3)\in \mathbb R^3$,  ${\bf w}=(w_1,w_2,w_3)\in \mathbb R^3.$\\

Authors in  \cite{li} introduced octonion linear canonical transform (OLCT)  as:\\
\begin{definition}[1D-OLCT \cite{li}]\label{def 1dolct}
  The 1D  OLCT of any octonion-valued  signal $f\in L^1(\mathbb R^3,\mathbb O),$ with respect to the uni-modular  matrix $\Lambda=(a,b,c,d)$ is given by
\begin{equation}\label{onedOLCT}
\mathcal L^{\Lambda}_{\tau_4}[f] (w)= \int_{\mathbb R} f(x)K_{\Lambda}^{\tau_4}(x,w)dx,
\end{equation}
where
 \begin{equation}\label{eqn 1D olct ker}
K^{\tau_4}_{\Lambda}(x,w)=\dfrac{1}{\sqrt{2\pi |b|}} e^{\frac{{\tau_4}}{2b}\big[ax^2-2xw+dw^2-\frac{\pi}{2}\big]},\quad b\neq0\,
\end{equation}
 with the inversion formula
 \begin{eqnarray}\label{eqn inv 1dolct}
 \nonumber f(x)&=&\{\mathcal L^{\Lambda}_{\tau_4}\}^{-1}[\mathcal L^{\Lambda}_{\tau_4}[f]](x)\\
  \nonumber&=&\int_{\mathbb R}\mathcal L^{\Lambda}_{\tau_4}\{f\}(w)K_{\Lambda}^{-\tau_4}(x,w)dx,\\
 \end{eqnarray}
 where $K_{\Lambda}^{-\tau_4}(x,w)=K_{\Lambda^{-1}}^{\tau_4}(w,x)$ and $\Lambda^{-1}=(d,-b,-c,a).$\\
 \end{definition}
\begin{lemma}\cite{li}\label{lem rel 1doft-1dolct}The 1D OLCT can be reduced to  1D OFT  by following equation:
\begin{equation}\label{rel 1doft-1dolct}
\mathcal L^{\Lambda}_{\tau_4}[f] (w)= \frac{1}{\sqrt{2\pi|b|}}\mathcal F_{\tau_4}[g]\left(\frac{w}{2\pi|b|}\right)e^{\tau_4\left(\frac{d}{2b}w^2-\frac{\pi}{4}\right)},
\end{equation} 
where $g(x)=f(x)e^{\tau_4\frac{a}{2b}x^2}.$
\end{lemma}

\begin{definition}[3D-OLCT \cite{li}]\label{def 3dolct}
For every octonion-valued signal  $f\in L^1(\mathbb R^3,\mathbb O)\cap L^2(\mathbb R^3,\mathbb O)$, the 3D-OLCT with respect to the   matrix parameters $\Lambda_k=(a_k,b_k,c_k,d_k), $ satisfying $det(\Lambda_k)=1,\quad k=1,2,3$ is defined as
 \begin{equation}\label{threedOLCT}
\mathcal L^{\Lambda_1,\Lambda_2,\Lambda_3}_{\tau_1,\tau_2,\tau_4}\{f\}({\bf w})=\int_{\mathbb R^3}f({\bf x}) K^{\tau_1}_{\Lambda_1}(x_1,w_1)K^{\tau_2}_{\Lambda_1}(x_2,w_2)K^{\tau_4}_{\Lambda_3}(x_3,w_3)d{\bf x}
\end{equation}
where ${\bf x}=(x_{1},x_{2},x_3),\, {\bf w}=(w_{1},w_{2},w_3),$ and  multiplication in above integral is done from left to right and
\begin{equation}\label{k1}
K_{\Lambda_1}^{\tau_1}(x_1,w_1)=\dfrac{1}{\sqrt{2\pi |b_1|}} e^{\frac{{\tau_1}}{2b_1}\big[a_1x^2_1-2x_1w_1+d_1w^2_1-\frac{\pi}{2}\big]},\quad b_1\neq0\,
\end{equation}
 \begin{equation}\label{k2}
K_{\Lambda_2}^{\tau_2}(x_2,w_2)==\dfrac{1}{\sqrt{2\pi |b_2|}} e^{\frac{{\tau_2}}{2b_2}\big[a_2x^2_2-2x_2w_2+d_2w^2_2-\frac{\pi}{2}\big]},\quad b_2\neq0\,
\end{equation}
and
 \begin{equation}\label{k3}
K_{\Lambda_3}^{\tau_4}(x_3,w_3)=\dfrac{1}{\sqrt{2\pi |b_3|}} e^{\frac{{\tau_4}}{2b_3}\big[a_3x^2_3-2x_3w_3+d_3w^2_3-\frac{\pi}{2}\big]},\quad b_3\neq0 .
\end{equation}
with the inversion formula
\begin{eqnarray}\label{eqn inv 3dolct}
\nonumber f({\bf x})&=&\{\mathcal L^{\Lambda_1,\Lambda_2,\Lambda_3}_{\tau_1,\tau_2,\tau_4}\}^{-1}\left[\mathcal L^{\Lambda_1,\Lambda_2,\Lambda_3}_{\tau_1,\tau_2,\tau_4}\{f\} \right]({\bf x})\\
\nonumber&=&\int_{\mathbb R^3}\mathcal L^{\Lambda_1,\Lambda_2,\Lambda_3}_{\tau_1,\tau_2,\tau_4}\{f\}(w)K_{\Lambda^{-1}_3}^{\tau_4}(w_3,x_3)K_{\Lambda^{-1}_2}^{\tau_2}(w_2,x_2)K_{\Lambda^{-1}_1}^{\tau_1}(w_1,x_1)d{\bf w},\\
\end{eqnarray}
where $\Lambda^{-1}_k=(d_k,-b_k,-c_k,a_k)\in\mathbb R^{2\times 2},$ for $ k=1,2,3.$
\end{definition}
Further authors in \cite{li} expanded the kernel of 3D-OLCT as:

\begin{eqnarray}\label{eulerproduct}
 \nonumber K_{\Lambda_1}^{\tau_1}(x_1,w_1) K_{\Lambda_2}^{\tau_2}(x_2,w_2) K_{\Lambda_3}^{\tau_4}(x_3,w_3)&=&\frac{1}{2\pi\sqrt{2\pi|b_1b_2b_3|}}e^{\tau_1\theta_1}e^{\tau_2\theta_2}e^{\tau_4\theta_3}\\
\nonumber&=&\frac{1}{2\pi\sqrt{2\pi|b_1b_2b_3|}}(c_1+\tau_1s_1)(c_2+\tau_2s_2)(c_3+\tau_4s_3)\\
\nonumber&=&\frac{1}{2\pi\sqrt{2\pi|b_1b_2b_3|}}(c_1c_2c_3+s_1c_2c_3\tau_1+c_1s_2c_3\tau_2\\
\nonumber&\quad+& s_1s_2c_3\tau_3+c_1c_2s_3\tau_4+s_1c_2s_3\tau_5+c_1s_2s_3\tau_6+s_1s_2s_3\tau_7),\\
\end{eqnarray}
where $\theta_k={\frac{{1}}{2b_k}\big[a_kx^2_k-2x_kw_k+d_kw^2_k-\frac{\pi}{2}\big]},\quad c_k=\cos\theta_k$ and $s_k=\sin\theta_k,\quad k=1,2,3.$\\

\section{  Wigner Distribution In The Octonion Linear Canonical Domain} \  \\\label{sec 3}
In this section we formally introduce the definition of  Wigner distribution associated with octonion linear canonical transform (WDOL)and study its various properties.

According to 1D-OLCT, we can obtain the definition of 1D-WDOL as follows
\begin{definition}[1D-WDOL]\label{def 1D WDOL} The 1D-WDOL of any octonion-valued  signals $f,g\in L^1(\mathbb R^3,\mathbb O),$ with respect to the uni-modular  matrix $\Lambda=(a,b,c,d)$ is given by
\begin{equation}\label{eqn 1D WDOL}
\mathbb W_{f,g}^{\tau_4}(t,w)= \int_{\mathbb R} f\left(t+\frac{x}{2}\right)g^*\left(t-\frac{x}{2}\right)K_{\Lambda}^{\tau_4}(x,w)dx,
\end{equation}
where  kernel $K_{\Lambda}^{\tau_4}(x,w)$ is given by (\ref{eqn 1D olct ker}).
\end{definition}
Next we note that 1D-WDOL is an 1D-OLCT of a instantaneous correlation of octonion-valued functions $f$ and $g$ by means of remark as:
\begin{remark}From definition of 1D-OLCT, it is clear that if we take \begin{equation}\label{fun h}
h_t(x)=f\left(t+\frac{x}{2}\right)g^*\left(t-\frac{x}{2}\right),
\end{equation}
we have 
\begin{eqnarray}\label{1d WDOL in h}
\nonumber\mathbb W_{f,g}^{\tau_4}(t,w)\\
\nonumber&=&\mathcal L^{\Lambda}_{\tau_4}\left[f\left(t+\frac{x}{2}\right)g^*\left(t-\frac{x}{2}\right)\right] (w)\\
\nonumber&=&\mathcal L^{\Lambda}_{\tau_4}[h_t] (w).\\
\end{eqnarray}
\end{remark}
And by using Lemma \ref{lem rel 1doft-1dolct}, we obtain the relation between 1D-WDOL and 1D-OFT as :
\begin{equation}\label{eqn rel 1doft-1dolct}
\mathbb W_{f,g}^{\tau_4}(t,w)=\frac{1}{\sqrt{2\pi |b|}}\mathcal F_{\tau_4}\left[H_t\right]\left(\frac{w}{2\pi |b|}\right)e^{\tau_4(\frac{d}{2b}w^2-\frac{\pi}{4})},
\end{equation}
where $H_t(x)=h_t(x)e^{\tau_4\frac{a}{2b}x^2}.$
\begin{theorem}[1D-WDOL Reconstruction Formula]
For $f,g\in L^2(\mathbb R, \mathbb O)$ with $g(0)\ne 0$, we have the following inversion formula for 1D-WDOL
\begin{eqnarray}
f\left(\xi\right)&=&\frac{1}{g^*\left(0\right)}\int_{\mathbb R}\mathbb W_{f,g}^{\tau_4}\left(\frac{\xi}{2},w\right)K_{\Lambda}^{-\tau_4}(\xi,w)dw
\end{eqnarray}
\end{theorem}
\begin{proof}
From \ref{eqn rel 1doft-1dolct}, we obtain
\begin{equation}
\sqrt{2\pi |b|}\mathbb W_{f,g}^{\tau_4}(t,w)e^{-\tau_4(\frac{d}{2b}w^2-\frac{\pi}{2})}=\mathcal F_{\tau_4}\left[H_t\right]\left(\frac{w}{2\pi |b|}\right),
\end{equation}
where $H_t(x)=f\left(t+\frac{x}{2}\right)g^*\left(t-\frac{x}{2}\right)e^{\tau_4\frac{a}{2b}x^2}.$\\
Now by the application of (\ref{eqn inv 1doft}), it follows that
\begin{eqnarray*}
&&f\left(t+\frac{x}{2}\right)g\left(t-\frac{x}{2}\right)e^{\tau_4\frac{a}{2b}x^2}\\
&&=\int_{\mathbb R}\sqrt{2\pi |b|}\mathbb W_{f,g}^{\tau_4}(t,w)e^{-\tau_4(\frac{d}{2b}w^2-\frac{\pi}{2})}e^{\tau_4x\frac{w}{b}} d\left(\frac{u}{2\pi |b|}\right)\\
\end{eqnarray*}
Which implies 
\begin{eqnarray}
\nonumber&&f\left(t+\frac{x}{2}\right)g^*\left(t-\frac{x}{2}\right)\\
\nonumber&&=\int_{\mathbb R}\mathbb W_{f,g}^{\tau_4}(t,w)\frac{1}{\sqrt{2\pi |b|}}e^{-\tau_4(\frac{d}{2b}w^2-\frac{\pi}{2})}e^{\tau_42 \pi x\frac{w}{2\pi b}}e^{-\tau_4\frac{a}{2b}x^2}dw\\
\label{r1}&&=\int_{\mathbb R}\mathbb W_{f,g}^{\tau_4}(t,w)K_{\Lambda}^{-\tau_4}(x,w)dw.
\end{eqnarray}
On setting $\frac{x}{2}=t,$ and applying change of variable $\xi=2t$, (\ref{r1}) yields
\begin{eqnarray*}
f\left(\xi\right)g^*\left(0\right)&=&\int_{\mathbb R}\mathbb W_{f,g}^{\tau_4}\left(\frac{\xi}{2},w\right)K_{\Lambda}^{-\tau_4}(\xi,w)dw,
\end{eqnarray*}
which completes the proof
\end{proof}

\begin{theorem}[Plancherel's Theorem] Let $f,g\in L^2(\mathbb R,\mathbb O),$ then the 1D-WDOL satisfies:
\begin{eqnarray}
\|\mathbb W^{\tau_4}_{f,g}\|^2_{L^2(\mathbb R,\mathbb O)}&=&
\|f\|^2_{L^2(\mathbb R,\mathbb O)}\|g\|^2_{L^2(\mathbb R, \mathbb O)}.
\end{eqnarray}
\end{theorem}
\begin{proof}We know that for $f,g\in L^2(\mathbb R,\mathbb O),$ we have
\begin{eqnarray*}
\langle f,g\rangle_{L^2(\mathbb R,\mathbb O)}&=&\int_{\mathbb R}f(x)g^*(x)dx\\
&=&\int_{\mathbb R}\left(\int_{\mathbb R}\mathcal L^\Lambda_{\tau_4}[f](w)K_{\Lambda}^{-\tau_4}(x,w)dw\right)g^*(x)dx\\
&=&\int_{\mathbb R}\mathcal L^\Lambda_{\tau_4}[f](w)\left(\int_{\mathbb R}g(x)K_{\Lambda}^{\tau_4}(x,w)dx\right)^*dw\\
&=&\int_{\mathbb R}\mathcal L^\Lambda_{\tau_4}[f](w)\left(\mathcal L^\Lambda_{\tau_4}[g](w)\right)^*dw\\
&=&\langle\mathcal L^\Lambda_{\tau_4}[f](w),\mathcal L^\Lambda_{\tau_4}[g](w)\rangle_{L^2(\mathbb R,\mathbb O)}
\end{eqnarray*}
Thus for $f=g$, above yields
\begin{equation*}
\|f\|^2_{L^2(\mathbb R,\mathbb O)}=\|\mathcal L^\Lambda_{\tau_4}[f]\|^2_{L^2(\mathbb R,\mathbb O)}
\end{equation*}
Replacing $f(x)$ by $h_t(x)$, we have
\begin{equation*}
\|h_t\|^2_{L^2(\mathbb R,\mathbb O)}=\|\mathcal L^\Lambda_{\tau_4}[h_t]\|^2_{L^2(\mathbb R,\mathbb O)}
\end{equation*}
Now applying (\ref{1d WDOL in h}), above eqn. yields
\begin{eqnarray*}
\|\mathbb W^{\tau_4}_{f,g}\|^2_{L^2(\mathbb R,\mathbb O)}&=&\left\|f\left(t+\frac{x}{2}\right)g^*\left(t-\frac{x}{2}\right)\right\|^2_{L^2(\mathbb R,\mathbb O)}\\
&=&\left(\int_{\mathbb R}\int_{\mathbb R}\left|f\left(t+\frac{x}{2}\right)g^*\left(t-\frac{x}{2}\right)\right|^2dtds\right)\\
&=&
\int_{\mathbb R}\int_{\mathbb R}\left|f\left(u\right)g^*\left(v\right)\right|^2dudv
\\
&=&
\int_{\mathbb R}\left|f\left(u\right)\right|^2du\int_{\mathbb R}\left|g^*\left(v\right)\right|^2dv
\\
&=&\|f\|^2_{L^2(\mathbb R,\mathbb O)}\|g\|^2_{L^2(\mathbb R, \mathbb O)}.
\end{eqnarray*}
Which completes the proof.
\end{proof}
By using the relationship between 1D-WDOL and 1D-OFT, we can prove the properties like $\mathbb R-$antilinearity, Scaling, Shift, Modulation of 1D-WDOLC following the procedure defined in \cite{1doft}.\\

Next, we will introduce the definition of  the 3D-WDOL.

\begin{definition}[3D-WDOL]\label{def 3d-WDOL}
Let $f,g:\mathbb R^3\rightarrow \mathbb O$ be two octonion-valued functions, then 3D-WDOL with respect to the   matrix parameters $\Lambda_k=(a_k,b_k,c_k,d_k), $ satisfying $det(\Lambda_k)=1,\quad k=1,2,3$ is defined as
\begin{equation}\label{eqn 3d WDOL}
\mathbb W ^{\Lambda_1,\Lambda_2,\Lambda_3}_{f,g}({\bf t, w})=\int_{\mathbb R^3}f\left(\bf{t+\frac{x}{2}}\right) g^*\left({\bf t-\frac{x}{2}}\right) K^{\tau_1}_{\Lambda_1}(x_1,w_1)K^{\tau_2}_{\Lambda_2}(x_2,w_2)K^{\tau_4}_{\Lambda_3}(x_3,w_3)d{\bf x}.
\end{equation}

where ${\bf x}=(x_{1},x_{2},x_3),\, {\bf w}=(w_{1},w_{2},w_3),\,{\bf t}=(t_1,t_2,t_3)$  and  kernel signals $K^{\tau_1}_{\Lambda_1}(x_1,w_1),$ $K^{\tau_2}_{\Lambda_2}(x_2,w_2),$ and $K^{\tau_4}_{\Lambda_3}(x_3,w_3)$ are given by (\ref{k1}),(\ref{k2}) and (\ref{k3}) respectively.

\end{definition}
It should be noted that the
multiplication in the above integrals is done from left to right 
as the octonion algebra is non-associative. Also we assume that the above signals $f,g$ are continuous and both signals and there  WDOL are integrable(in Lebesgue sense) in this paper.\\

Next we note that 3D-WDOL is an 3D-OLCT of a instantaneous correlation of 3D octonion-valued functions $f$ and $g$ by means of remark as:
\begin{remark}From definition of 3D-OLCT, it is clear that if we take \begin{equation}\label{fun 3d h}
h_{\bf t}({\bf x})=f\left(\bf{t+\frac{x}{2}}\right) g^*\left({\bf t-\frac{x}{2}}\right) ,
\end{equation}
we have
\begin{equation}\label{3d WDOL in h}
\mathbb W ^{\Lambda_1,\Lambda_2,\Lambda_3}_{f,g}({\bf t, w})=\mathcal L^{\Lambda_1,\Lambda_2,\Lambda_3}_{\tau_1,\tau_2,\tau_4}[h_{\bf t}] ({\bf w}).\\
\end{equation}
\end{remark}

The 3D-WDOL $\mathbb W ^{\Lambda_1,\Lambda_2,\Lambda_3}_{f,g}({\bf t, w})$ defined in (\ref{eqn 3d WDOL}) can be expressed 
 as octonion sum of components of different parity by using (\ref{eulerproduct}) as:

\begin{eqnarray}\label{parity}
\nonumber\mathbb W^{\Lambda_1,\Lambda_2,\Lambda_3}_{f,g}({\bf t, w})&=&W^{eee}_{f,g}+W^{oee}_{f,g}\tau_1+W^{eoe}_{f,g}\tau_2+W^{ooe}_{f,g}\tau_3+G^{eeo}_{f,g}\tau_4\\
\nonumber&&\qquad\qquad\qquad+W^{oeo}_{f,g}\tau_5+W^{eoo}_{f,g}\tau_6+W^{ooo}_{f,g}\tau_7\\\
\end{eqnarray}
where
\begin{equation}\label{e1}
W^{eee}_{f,g}({\bf t, w})=\frac{1}{2\pi\sqrt{2\pi|b_1b_2b_3|}}\int_{\mathbb R^3}h_{\bf t}^{eee}({\bf x})c_1c_2c_3d{\bf x},
\end{equation}
\begin{equation}\label{e2}
W^{oee}_{f,g}({\bf t, w}))=\frac{1}{2\pi\sqrt{2\pi|b_1b_2b_3|}}\int_{\mathbb R^3}h_{oee}({\bf x})s_1c_2c_3d{\bf x},
\end{equation}

\begin{equation}\label{e3}
W^{eoe}_{f,g}({\bf t, w}))=\frac{1}{2\pi\sqrt{2\pi|b_1b_2b_3|}}\int_{\mathbb R^3}h_{\bf t}^{eoe}({\bf x})c_1s_2c_3d{\bf x},
\end{equation}

\begin{equation}\label{e4}
W^{ooe}_{f,g}({\bf t, w}))=\frac{1}{2\pi\sqrt{2\pi|b_1b_2b_3|}}\int_{\mathbb R^3}h_{\bf t}^{ooe}({\bf x})s_1s_2c_3d{\bf x},
\end{equation}

\begin{equation}\label{e5}
W^{eeo}_{f,g}({\bf t, w}))=\frac{1}{2\pi\sqrt{2\pi|b_1b_2b_3|}}\int_{\mathbb R^3}h_{\bf t}^{eeo}({\bf x})c_1c_2s_3d{\bf x},
\end{equation}

\begin{equation}\label{e6}
W^{oeo}_{f,g}({\bf t, w}))=\frac{1}{2\pi\sqrt{2\pi|b_1b_2b_3|}}\int_{\mathbb R^3}h_{\bf t}^{oeo}({\bf x})s_1c_2s_3d{\bf x},
\end{equation}

\begin{equation}\label{e7}
W^{eoo}_{f,g}({\bf t, w}))=\frac{1}{2\pi\sqrt{2\pi|b_1b_2b_3|}}\int_{\mathbb R^3}h_{\bf t}^{eoo}({\bf x})c_1c_2s_3d{\bf x},
\end{equation}

\begin{equation}\label{e8}
W^{ooo}_{f,g}({\bf t, w}))=\frac{1}{2\pi\sqrt{2\pi|b_1b_2b_3|}}\int_{\mathbb R^3}h_{\bf t}^{ooo}({\bf x})s_1s_2s_3d{\bf x}.
\end{equation}

 Where $h_{\bf t}({\bf x})=f\left(\bf{t+\frac{x}{2}}\right) g^*\left({\bf t-\frac{x}{2}}\right)$  and can be expressed as  sum  eight terms:
 \begin{eqnarray}\label{h even odd}
 \nonumber h_{\bf t}({\bf x})&=&h_{\bf t}^{eee}({\bf x})+h_{\bf t}^{eeo}({\bf x})+h_{\bf t}^{eoe}({\bf x})+h_{\bf t}^{eoo}({\bf x})\\\
 \nonumber&&+h_{\bf t}^{oee}({\bf x})+h_{\bf t}^{oeo}({\bf x})+h_{\bf t}{ooe}({\bf x})+h_{\bf t}^{ooo}({\bf x}),\\\
 \end{eqnarray}
where $h_{\bf t}^{lmn}({\bf x}),l,m,n\in\{e,o\}$ are eight terms of different parity with relation to $x_1,
x_2$ and $x_3.$ Again using subscripts $e$ and $o$ to indicate that a
function is either even (e) or odd (o) with respect to an appropriate
variable, i.e. $h_{\bf t}^{eeo}({\bf x})$ is even with respect to $x_1$ and $x_2$ and odd with
respect to $x_3$.\\

Now, we show that 3D-WDOL can be divided into four Wigner distributions in the  QLCT domain (WDQLCT).
\begin{lemma}[Relation with WDQLCT]\label{rel with WDQLCT}For $f,g\in L^2(\mathbb R^3, \mathbb O)$  and $\mathbb W ^{\Lambda_1,\Lambda_2,\Lambda_3}_{f,g}({\bf t, w})$ be the 3D-WDOL, then
\begin{eqnarray}
\nonumber\left|\mathbb W ^{\Lambda_1,\Lambda_2,\Lambda_3}_{f,g}({\bf t, w})\right|^2
&=&\frac{1}{2\pi b_3}\left(\left|\mathbb W^{\Lambda_1,\Lambda_2}_{(\tilde f,\tilde g)^e} ({\bf t,w})\right|^2+\left|\mathbb W^{\Lambda_1,\Lambda_2}_{(\hat f,\hat g)^o} ({\bf t, w})\right|^2\right.\\
\label{eqn rel WDQLCT}&&+\left.\left|\mathbb W^{\Lambda_1,\Lambda_2}_{(\hat f,\hat g)^e} ({\bf t, w})\right|^2+\left|\mathbb W^{\Lambda_1,\Lambda_2}_{(\tilde f,\tilde g)^o} ({\bf t, w})\right|^2\right),
\end{eqnarray}
where $\tilde f,\, \hat f,\, \tilde g,\,\hat g\in L^2(\mathbb R^2,\mathbb H).$
\end{lemma}
\begin{proof}
Let $h_{\bf t}^e({\bf x})=\frac{1}{2}[h_{\bf t}(x_1,x_2,x_3)+h_{\bf t}(x_1,x_2,-x_3)],$  and $h_{\bf t}^o({\bf x})=\frac{1}{2}[h_{\bf t}(x_1,x_2,x_3)-h_{\bf t}(x_1,x_2,-x_3)],$  then $h_{\bf t}^e({\bf x})$ and $h_{\bf t}^o({\bf x})$ represents   even and odd parts of $h_{\bf t}({\bf x})$ but only in variable $x_3.$\\
 Since every octonion function can be written in the quaternion form as  $h_{\bf t}({\bf x})=\tilde h_{\bf t}+\hat {h_{\bf t}}\tau_4$, therefore from (\ref{3d WDOL in h}), we have
 \begin{eqnarray}
\nonumber\mathbb W ^{\Lambda_1,\Lambda_2,\Lambda_3}_{f,g}({\bf t, w})\\
\nonumber&=&\mathcal L^{\Lambda_1,\Lambda_2,\Lambda_3}_{\tau_1,\tau_2,\tau_4}[\tilde h_{\bf t}+\hat {h_{\bf t}}\tau_4] ({\bf w})\\
\nonumber&=&\int_{\mathbb R^3}(\tilde h_{\bf t}+\hat {h_{\bf t}}\tau_4) K^{\tau_1}_{\Lambda_1}(x_1,w_1)K^{\tau_2}_{\Lambda_2}(x_2,w_2)K^{\tau_4}_{\Lambda_3}(x_3,w_3)d{\bf x}\\
\nonumber&=&\int_{\mathbb R^3}\tilde h_{\bf t}({\bf x}) K^{\tau_1}_{\Lambda_1}(x_1,w_1)K^{\tau_2}_{\Lambda_2}(x_2,w_2)K^{\tau_4}_{\Lambda_3}(x_3,w_3)d{\bf x}\\
\label{a1}&&+\int_{\mathbb R^3}\hat {h_{\bf t}}({\bf x}) K^{-\tau_1}_{\Lambda_1}(x_1,w_1)K^{-\tau_2}_{\Lambda_2}(x_2,w_2)\tau_4K^{\tau_4}_{\Lambda_3}(x_3,w_3)d{\bf x}.
\end{eqnarray}
Now taking even and odd parts of functions  $\tilde h_{\bf t}$ and $\hat {h_{\bf t}},$ (\ref{a1}) yields
\begin{eqnarray}
\nonumber\mathbb W ^{\Lambda_1,\Lambda_2,\Lambda_3}_{f,g}({\bf t, w})\\
\nonumber&=&\frac{1}{\sqrt{2\pi b_3}}\int_{\mathbb R^3}\tilde h^e_{\bf t}({\bf x}) K^{\tau_1}_{\Lambda_1}(x_1,w_1)K^{\tau_2}_{\Lambda_2}(x_2,w_2)c_3d{\bf x}\\
\nonumber&&+\frac{1}{\sqrt{2\pi b_3}}\int_{\mathbb R^3}\hat {h^o_{\bf t}}({\bf x}) K^{-\tau_1}_{\Lambda_1}(x_1,w_1)K^{-\tau_2}_{\Lambda_2}(x_2,w_2)s_3d{\bf x}\\
\nonumber &&+\left(\frac{1}{\sqrt{2\pi b_3}}\int_{\mathbb R^3}\hat {h^e_{\bf t}}({\bf x}) K^{-\tau_1}_{\Lambda_1}(x_1,w_1)K^{-\tau_2}_{\Lambda_2}(x_2,w_2)c_3d{\bf x}\right.\\
\nonumber&&\left.-\frac{1}{\sqrt{2\pi b_3}}\int_{\mathbb R^3} \tilde h^o_{\bf t}({\bf x}) K^{\tau_1}_{\Lambda_1}(x_1,w_1)K^{\tau_2}_{\Lambda_2}(x_2,w_2)s_3d{\bf x}\right)\tau_4,\\
\end{eqnarray}
where $\theta_k={\frac{{1}}{2b_k}\big[a_kx^2_k-2x_kw_k+d_kw^2_k-\frac{\pi}{2}\big]},\quad c_k=\cos\theta_k$ and $s_k=\sin\theta_k,\quad k=1,2,3.$\\
Thus
\begin{eqnarray}
\nonumber\left|\mathbb W ^{\Lambda_1,\Lambda_2,\Lambda_3}_{f,g}({\bf t, w})\right|^2\\
\nonumber&=&\frac{1}{2\pi b_3}\left(\left|\mathcal L^{\Lambda_1,\Lambda_2}_{\tau_1,\tau_2}[\tilde h^e_{\bf t}] ({\bf w})\right|^2+\left|\mathcal L^{\Lambda_1,\Lambda_2}_{\tau_1,\tau_2}[\hat h^0_{\bf t}] ({\bf w})\right|^2\right.\\
\label{a3}&&+\left.\left|\mathcal L^{\Lambda_1,\Lambda_2}_{\tau_1,\tau_2}[\hat h^e_{\bf t}] ({\bf w})\right|^2+\left|\mathcal L^{\Lambda_1,\Lambda_2}_{\tau_1,\tau_2}[\tilde h^o_{\bf t}] ({\bf w})\right|^2\right),
\end{eqnarray}
where the terms in RHS denote the quaternion LCT of the respective  correlation
product.\\
Applying the definition of WD in the QLCT domain, (\ref{a3}) gives
\begin{eqnarray*}
\nonumber\left|\mathbb W ^{\Lambda_1,\Lambda_2,\Lambda_3}_{f,g}({\bf t, w})\right|^2\\
\nonumber&=&\frac{1}{2\pi b_3}\left(\left|\mathbb W^{\Lambda_1,\Lambda_2}_{(\tilde f,\tilde g)^e} ({\bf t,w})\right|^2+\left|\mathbb W^{\Lambda_1,\Lambda_2}_{(\hat f,\hat g)^o} ({\bf t, w})\right|^2\right.\\
&&+\left.\left|\mathbb W^{\Lambda_1,\Lambda_2}_{(\hat f,\hat g)^e} ({\bf t, w})\right|^2+\left|\mathbb W^{\Lambda_1,\Lambda_2}_{(\tilde f,\tilde g)^o} ({\bf t, w})\right|^2\right).
\end{eqnarray*}
Which completes the proof.
\end{proof}

\begin{theorem}[Reconstruction formula of the 3D-WDOL]If $f,g\in L^2(\mathbb R^3,\mathbb O)$ and $g(0)\ne0,$ then $f$ can be reconstructed by the inverse 3D-OLCT of 3D-WDOL $\mathbb W^{\Lambda_1,\Lambda_2,\Lambda_3}_{f,g},$ i.e.
\begin{eqnarray*}
f\left(\bf{x}\right) 
&=&\frac{1}{g^*\left(0\right)}\int_{\mathbb R^3}\mathbb W^{\Lambda_1,\Lambda_2,\Lambda_3}_{f,g}\left({\bf \frac{x}{2},w}\right)\overline{K_{\Lambda_3}^{\tau_4}(w_3,x_3)}.\overline{K_{\Lambda_2}^{\tau_2}(w_2,x_2)}.\overline{K_{\Lambda_1}^{\tau_1}(w_1,x_1)}d{\bf w}.\\
\end{eqnarray*}
\end{theorem}
\begin{proof}
Applying 3D-OLCT inversion given in (\ref{eqn inv 3dolct}) to (\ref{3d WDOL in h}), we have
\begin{eqnarray*}
h_{\bf t}({\bf x})&=&\{\mathcal L^{\Lambda_1,\Lambda_2,\Lambda_3}_{\tau_1,\tau_2,\tau_4}\}^{-1}\left[\mathbb W^{\Lambda_1,\Lambda_2,\Lambda_3}_{f,g}\right]({\bf x})\\
\nonumber&=&\int_{\mathbb R^3}\mathbb W^{\Lambda_1,\Lambda_2,\Lambda_3}_{f,g}({\bf t,w})K_{\Lambda^{-1}_3}^{\tau_4}(w_3,x_3)K_{\Lambda^{-1}_2}^{\tau_2}(w_2,x_2)K_{\Lambda^{-1}_1}^{\tau_1}(w_1,x_1)d{\bf w}.\\
\end{eqnarray*}
Hence
\begin{eqnarray*}
&&f\left(\bf{t+\frac{x}{2}}\right) g^*\left({\bf t-\frac{x}{2}}\right)\\
&&=\int_{\mathbb R^3}\mathbb W^{\Lambda_1,\Lambda_2,\Lambda_3}_{f,g}({\bf t,w})\overline{K_{\Lambda_3}^{\tau_4}(w_3,x_3)}.\overline{K_{\Lambda_2}^{\tau_2}(w_2,x_2)}.\overline{K_{\Lambda_1}^{\tau_1}(w_1,x_1)}d{\bf w}.\\
\end{eqnarray*}
Setting ${\bf t=\frac{x}{2}},$ above equation yields
\begin{eqnarray*}
&&f\left(2\bf{t}\right) g^*\left(0\right)\\
&&=\int_{\mathbb R^3}\mathbb W^{\Lambda_1,\Lambda_2,\Lambda_3}_{f,g}({\bf t,w})\overline{K_{\Lambda_3}^{\tau_4}(w_3,2t_3)}.\overline{K_{\Lambda_2}^{\tau_2}(w_2,2t_2)}.\overline{K_{\Lambda_1}^{\tau_1}(w_1,2t_1)}d{\bf w}.\\
\end{eqnarray*}
Applying the change of variable ${\bf y}=2{\bf t},$ we get
\begin{eqnarray*}
f\left(\bf{y}\right)
&=&\frac{1}{g^*\left(0\right)}\int_{\mathbb R^3}\mathbb W^{\Lambda_1,\Lambda_2,\Lambda_3}_{f,g}\left({\bf \frac{y}{2},w}\right)\overline{K_{\Lambda_3}^{\tau_4}(w_3,y_3)}.\overline{K_{\Lambda_2}^{\tau_2}(w_2,y_2)}.\overline{K_{\Lambda_1}^{\tau_1}(w_1,y_1)}d{\bf w}.\\
\end{eqnarray*}
Which completes the proof.
\end{proof}

\begin{theorem}[Rayleigh’s theorem for 3D-WDOL]\label{plan 3D WDOL} Let $f,g\in L^2(\mathbb R^3,\mathbb O)$, then we have
\begin{equation}\label{eqn plan 3d vdol}
2\pi|b_3|\|\mathbb W ^{\Lambda_1,\Lambda_2,\Lambda_3}_{f,g}({\bf t, w})\|^2_{L^2(\mathbb R^2,\mathbb O)}=\|f\|^2_{L^2(\mathbb R^2,\mathbb O)}.\|g\|^2_{L^2(\mathbb R^2,\mathbb O)}.
\end{equation}
\end{theorem}
\begin{proof}
The Rayleigh’s theorem is valid for the 3D-OLCT and reads (see Thm.4 \cite{li})
\begin{equation*}
2\pi|b_3\left|\mathcal L^{\Lambda_1,\Lambda_2,\Lambda_3}_{\tau_1,\tau_2,\tau_4}[f]\right\|^2_{L^2(\mathbb R^2,\mathbb O)}=\|f\|^2_{L^2(\mathbb R^2,\mathbb O)}
\end{equation*}
Replacing $f({\bf x})$ by $h_{\bf t}({\bf x})$ given in (\ref{fun 3d h}), above equation becomes
\begin{equation*}
2\pi|b_3\left|\mathcal L^{\Lambda_1,\Lambda_2,\Lambda_3}_{\tau_1,\tau_2,\tau_4}[h_{\bf t}]\right\|^2_{L^2(\mathbb R^2,\mathbb O)}=\|h_{\bf t}\|^2_{L^2(\mathbb R^2,\mathbb O)}
\end{equation*}

Applying (\ref{3d WDOL in h}), above equation yields
 \begin{eqnarray*}
2\pi|b_3\left\|\mathbb W ^{\Lambda_1,\Lambda_2,\Lambda_3}_{f,g}({\bf t, w})\right\|^2_{L^2(\mathbb R^2,\mathbb O)}
&=&\left\|f\left({\bf t+\frac{x}{2}}\right)g^*\left({\bf t-\frac{x}{2}}\right)\right\|^2_{L^2(\mathbb R^2,\mathbb O)}\\
&=&\int_{\mathbb R^3}\int_{\mathbb R^3}\left|f\left({\bf t+\frac{x}{2}}\right)g^*\left({\bf t-\frac{x}{2}}\right)\right|^2d{\bf t}d{\bf x}.
\end{eqnarray*}
Applying the Fubini theorem and using suitable change of variables(just like earlier), we obtain
\begin{equation*}
2\pi|b_3|\|\mathbb W ^{\Lambda_1,\Lambda_2,\Lambda_3}_{f,g}({\bf t, w})\|^2_{L^2(\mathbb R^2,\mathbb O)}=\|f\|^2_{L^2(\mathbb R^2,\mathbb O)}.\|g\|^2_{L^2(\mathbb R^2,\mathbb O)}.
\end{equation*}
Which completes the proof.
\end{proof}

\begin{theorem} Let $f,g\in L^2(\mathbb R^3;\mathbb O),$  then the Riemann–Lebesgue lemma associated with 3D-WDOL holds with respect to ${\bf w},$ i.e., 
\begin{equation}
\lim_{|{\bf w}|\rightarrow 0}\left|\mathbb W ^{\Lambda_1,\Lambda_2,\Lambda_3}_{f,g}({\bf t, w})\right|\rightarrow 0, \quad {\bf w,x}\in \mathbb R^3.
\end{equation}
\end{theorem}
\begin{proof}
Since  Riemann–Lebesgue lemma holds for 3D-OLCT (see \cite{li}), therefore using relation between 3D-OLCT and 3D-WDOL the proof of theorem follows.
\end{proof}

Now, we shall establish  the relation between  3D-WDOL and  3D-WDLCT.

\begin{theorem}\label{th 3dslct rel}Let $f,g:\rightarrow \mathbb O$ be two octonion-valued signals and   $\mathcal W^{\Lambda_1,\Lambda_2,\Lambda_3}_{f,g}({\bf t,w})$ represents the 3D-WDLCT. Then the following equation holds
\begin{eqnarray}\label{eqn 3dstlct rel}
\nonumber  \mathbb W ^{\Lambda_1,\Lambda_2,\Lambda_3}_{f,g}({\bf t, w})&=&\dfrac{1}{4}\left\{(\mathcal W^{\Lambda_1,\Lambda_2,\Lambda_3}_{f,g}({\bf t,w})+\mathcal W^{\Lambda_1,\Lambda_2,\Lambda'_3}_{f,g}({\bf t,w})(1-\tau_3)\right.\\
\nonumber\qquad &&\left.+(\mathcal W^{\Lambda_1,\Lambda'_2,\Lambda_3}_{f,g}({\bf t,w})+\mathcal W^{\Lambda_1,\Lambda'_2,\Lambda'_3}_{f,g}({\bf t,w})(1+\tau_3)\right\}\\
\nonumber\qquad\quad&& +\dfrac{1}{4}\left\{(\mathcal W^{\Lambda_1,\Lambda_2,\Lambda_3}_{f,g}({\bf t,w})-\mathcal W^{\Lambda_1,\Lambda_2,\Lambda'_3}_{f,g}({\bf t,w})(1-\tau_3)\right.\\
\nonumber\qquad &&\left.+(\mathcal W^{\Lambda_1,\Lambda'_2,\Lambda_3}_{f,g}({\bf t,w})-\mathcal W^{\Lambda_1,\Lambda'_2,\Lambda'_3}_{f,g}({\bf t,w})(1+\tau_3)\right\}.\tau_5\\
\end{eqnarray}
where $\Lambda'_k=(a_k,-b_k,-c_k,d_k),\quad k=2,3.$
\end{theorem}
\begin{proof}

For  $f,g\in L^2(\mathbb R^3)$ the 3D Wigner distribution associated with LCT (3D-WDLCT) is defined  corresponding to the 3D-LCT \cite{li} as:\\
   \begin{equation}\label{eqn b1}
\mathcal W^{\Lambda_1,\Lambda_2,\Lambda_3}_{f,g}({\bf t,w})=\dfrac{1}{2\pi\sqrt{2\pi|b_1b_2b_3|}}\int_{\mathbb R^3}f\left({\bf t+\frac{x}{2}}\right)g^*\left({\bf t-\frac{x}{2}}\right)e^{\tau_1\theta_1}e^{\tau_1\theta_2}e^{\tau_1\theta_3}d{\bf x}.
\end{equation}

Now for $\Lambda'_2=(a_2,-b_2,-c_2,d_2),$ then
\begin{equation}\label{eqn b2}
\mathcal W^{\Lambda_1,\Lambda'_2,\Lambda_3}_{f,g}({\bf t,w})=\dfrac{1}{2\pi\sqrt{2\pi|b_1b_2b_3|}}\int_{\mathbb R^3}f\left({\bf t+\frac{x}{2}}\right)g^*\left({\bf t-\frac{x}{2}}\right)e^{\tau_1\theta_1}e^{-\tau_1\theta_2}e^{\tau_1\theta_3}d{\bf x}.
\end{equation}
By equivalent definition of sine and cosine functions, we obtain
\begin{eqnarray}\label{b3}
\nonumber&&\frac{1}{2}\left(\mathcal W^{\Lambda_1,\Lambda_2,\Lambda_3}_{f,g}({\bf t,w})+\mathcal W^{\Lambda_1,\Lambda'_2,\Lambda_3}_{f,g}({\bf t,w})\right)\\ \
\nonumber&&\qquad\qquad=\dfrac{1}{2\pi\sqrt{2\pi|b_1b_2b_3|}}\int_{\mathbb R^3}f\left({\bf t+\frac{x}{2}}\right)g^*\left({\bf t-\frac{x}{2}}\right)e^{\tau_1\theta_1}c_2e^{\tau_1\theta_3}d{\bf x}.\\ \
\end{eqnarray}
And
\begin{eqnarray}\label{b4}
\nonumber&&\frac{1}{2}\left(\mathcal W^{\Lambda_1,\Lambda'_2,\Lambda_3}_{f,g}({\bf t,w})-\mathcal W^{\Lambda_1,\Lambda_2,\Lambda_3}_{f,g}({\bf t,w})\right)\\ \
\nonumber&&\qquad\qquad=\dfrac{1}{2\pi\sqrt{2\pi|b_1b_2b_3|}}\int_{\mathbb R^3}f\left({\bf t+\frac{x}{2}}\right)g^*\left({\bf t-\frac{x}{2}}\right)e^{\tau_1\theta_1}(-\tau_1s_2)e^{\tau_1\theta_3}d{\bf x}.\\ \
\end{eqnarray}
 Multiplying $\tau_3$ to (\ref{b4}) from right and using multiplication rules from Table \ref{table}, we have
\begin{eqnarray}\label{b5}
\nonumber&&\frac{1}{2}\left(\mathcal W^{\Lambda_1,\Lambda'_2,\Lambda_3}_{f,g}({\bf t,w})-\mathcal W^{\Lambda_1,\Lambda_2,\Lambda_3}_{f,g}({\bf t,w})\right)\tau_3\\ \
\nonumber&&\qquad\qquad=\dfrac{1}{2\pi\sqrt{2\pi|b_1b_2b_3|}}\int_{\mathbb R^3}f\left({\bf t+\frac{x}{2}}\right)g^*\left({\bf t-\frac{x}{2}}\right)e^{\tau_1\theta_1}(\tau_2s_2)e^{-\tau_1\theta_3}dx.\\ \
\end{eqnarray}
Adding (\ref{b3}) and (\ref{b5}), we have
\begin{eqnarray}\label{b6}
\nonumber&&\frac{1}{2}\left(\mathcal W^{\Lambda_1,\Lambda_2,\Lambda_3}_{f,g}({\bf t,w})+\mathcal W^{\Lambda_1,\Lambda'_2,\Lambda_3}_{f,g}({\bf t,w})\right)\\ \
\nonumber&&+\frac{1}{2}\left(\mathcal W^{\Lambda_1,\Lambda'_2,\Lambda_3}_{f,g}({\bf t,w})-\mathcal W^{\Lambda_1,\Lambda_2,\Lambda_3}_{f,g}({\bf t,w})\right)\tau_3\\ \
\nonumber&&\qquad\qquad=\dfrac{1}{2\pi\sqrt{2\pi|b_1b_2b_3|}}\int_{\mathbb R^3}f\left({\bf t+\frac{x}{2}}\right)g^*\left({\bf t-\frac{x}{2}}\right)e^{\tau_1\theta_1}e^{\tau_2\theta_2}e^{-\tau_1\theta_3}dx.\\ \
\end{eqnarray}
Let us  introduce new notation for simplification:
\begin{eqnarray}\label{b7}
\nonumber\textsf{ W}^{\frac{\Lambda_1,\Lambda_2,\Lambda_3}{\Lambda_1,\Lambda'_2,\Lambda_3}}_{f,g}({\bf t,w})&=&\frac{1}{2}\left(\mathcal W^{\Lambda_1,\Lambda_2,\Lambda_3}_{f,g}({\bf t,w})+\mathcal W^{\Lambda_1,\Lambda'_2,\Lambda_3}_{f,g}({\bf t,w})\right)\\ \
\nonumber&&+\frac{1}{2}\left(\mathcal W^{\Lambda_1,\Lambda'_2,\Lambda_3}_{f,g}({\bf t,w})-\mathcal W^{\Lambda_1,\Lambda_2,\Lambda_3}_{f,g}({\bf t,w})\right)\tau_3.\\ \
\end{eqnarray}
Now for $\Lambda'_3=(a_3,-b_3,-c_3,d_3),$ then
\begin{eqnarray}
\nonumber\textsf W^{\frac{\Lambda_1,\Lambda_2,\Lambda'_3}{\Lambda_1,\Lambda'_2,\Lambda'_3}}_{f,g}({\bf t,w})&=&\frac{1}{2}\left(\mathcal W^{\Lambda_1,\Lambda_2,\Lambda'_3}_{f,g}({\bf t,w})+\mathcal W^{\Lambda_1,\Lambda'_2,\Lambda'_3}_{f,g}({\bf t,w})\right)\\
\label{b8}&&+\frac{1}{2}\left(\mathcal W^{\Lambda_1,\Lambda'_2,\Lambda'_3}_{f,g}({\bf t,w})-\mathcal W^{\Lambda_1,\Lambda_2,\Lambda'_3}_{f,g}({\bf t,w})\right)\tau_3.\\
\label{b9}&=&\dfrac{1}{2\pi\sqrt{2\pi|b_1b_2b_3|}}\int_{\mathbb R^3}f\left({\bf t+\frac{x}{2}}\right)g^*\left({\bf t-\frac{x}{2}}\right)e^{\tau_1\theta_1}e^{\tau_2\theta_2}e^{\tau_3\theta_3}dx.
\end{eqnarray}
By following similar steps as before we get
\begin{eqnarray}\label{b10}
\nonumber&&\frac{1}{2}\left(\textsf W^{\frac{\Lambda_1,\Lambda_2,\Lambda_3}{\Lambda_1,\Lambda'_2,\Lambda_3}}_{f,g}({\bf t,w})+\textsf W^{\frac{\Lambda_1,\Lambda_2,\Lambda'_3}{\Lambda_1,\Lambda'_2,\Lambda'_3}}_{f,g}({\bf t,w})\right)\\ \
\nonumber&&\qquad\qquad=\dfrac{1}{2\pi\sqrt{2\pi|b_1b_2b_3|}}\int_{\mathbb R^3}f\left({\bf t+\frac{x}{2}}\right)g^*\left({\bf t-\frac{x}{2}}\right)e^{\tau_1\theta_1}e^{\tau_2\theta_2}c_3d{\bf x}.\\ \
\end{eqnarray}
And
\begin{eqnarray}\label{b11}
\nonumber&&\frac{1}{2}\left(\textsf W^{\frac{\Lambda_1,\Lambda_2,\Lambda_3}{\Lambda_1,\Lambda'_2,\Lambda_3}}_{f,g}({\bf t,w}))-\textsf W^{\frac{\Lambda_1,\Lambda_2,\Lambda'_3}{\Lambda_1,\Lambda'_2,\Lambda'_3}}_{f,g}({\bf t,w})\right)\\ \
\nonumber&&\qquad\qquad=\dfrac{1}{2\pi\sqrt{2\pi|b_1b_2b_3|}}\int_{\mathbb R^3}f\left({\bf t+\frac{x}{2}}\right)g^*\left({\bf t-\frac{x}{2}}\right)e^{\tau_1\theta_1}e^{\tau_2\theta_2}(-\tau_1s_3)dx.\\ \
\end{eqnarray}
On multiplying (\ref{b11}) from right by $\tau_5$ and using multiplication rules from Table \ref{table}, we have
\begin{eqnarray}\label{b12}
\nonumber&&\frac{1}{2}\left(\textsf W^{\frac{\Lambda_1,\Lambda_2,\Lambda_3}{\Lambda_1,\Lambda'_2,\Lambda_3}}_{f,g}({\bf t,w})-\textsf W^{\frac{\Lambda_1,\Lambda_2,\Lambda'_3}{\Lambda_1,\Lambda'_2,\Lambda'_3}}_{f,g}({\bf t,w})\right)\tau_5\\ \
\nonumber&&\qquad\qquad=\dfrac{1}{2\pi\sqrt{2\pi|b_1b_2b_3|}}\int_{\mathbb R^3}f\left({\bf t+\frac{x}{2}}\right)g^*\left({\bf t-\frac{x}{2}}\right)e^{\tau_1\theta_1}e^{\tau_2\theta_2}(\tau_4s_3)dx.\\ \
\end{eqnarray}
Adding (\ref{b10}) and (\ref{b12}), we get
\begin{eqnarray}\label{b13}
\nonumber&&\frac{1}{2}\left(\textsf W^{\frac{\Lambda_1,\Lambda_2,\Lambda_3}{\Lambda_1,\Lambda'_2,\Lambda_3}}_{f,g}({\bf t,w})+\textsf W^{\frac{\Lambda_1,\Lambda_2,\Lambda'_3}{\Lambda_1,\Lambda'_2,\Lambda'_3}}_{f,g}({\bf t,w})\right)\\ \
\nonumber&&+\frac{1}{2}\left(\textsf W^{\frac{\Lambda_1,\Lambda_2,\Lambda_3}{\Lambda_1,\Lambda'_2,\Lambda_3}}_{f,g}({\bf t,w})-\textsf W^{\frac{\Lambda_1,\Lambda_2,\Lambda'_3}{\Lambda_1,\Lambda'_2,\Lambda'_3}}_{f,g}({\bf t,w})\right)\tau_5\\ \
\nonumber&&\qquad\qquad=\dfrac{1}{2\pi\sqrt{2\pi|b_1b_2b_3|}}\int_{\mathbb R^3}f\left({\bf t+\frac{x}{2}}\right)g^*\left({\bf t-\frac{x}{2}}\right)e^{\tau_1\theta_1}e^{\tau_2\theta_2}e^{\tau_4\theta_3}dx.\\ \
\end{eqnarray}
On substituting (\ref{b7}) and (\ref{b8}) in (\ref{b13}), we get the desired result.
\end{proof}

\section{Uncertainty principles of the WDOL}\label{sec 4}
The uncertainty principles (UPs) lies in the heart of any integral transforms. In \cite{li,ownoolct} authors derived Heisenberg’s uncertainty principle and Donoho–Stark’s uncertainty principle,  Hausdorff–Young inequality,  logarithmic uncertainty inequality, Pitt’s inequality and local uncertainty inequality for the octonion linear canonical
transform and octonion offset linear canonical transform. Recently, in \cite{stoft,ownstolct} extend these UPs to the short-time octonion Fourier transform and short-time octonion linear canonical transform.
Considering the WDOL as an extension of WD-QLCT, so in this section we shall investigate some
uncertainty inequalities for the STOLCT. 

Lets begin with the Heisenberg’s uncertainty principle for the WDOL.
\begin{theorem}[Heisenberg’s uncertainty principle for the WDOL]Let $f,g\in L^1(\mathbb R^3,\mathbb O)\cap L^2(\mathbb R^3,\mathbb O),$ then 3D-WDOL satisfies following inequality
\begin{eqnarray}\label{hes}
\nonumber&&\left(\int_{\mathbb R^2}\int_{\mathbb R^2}{\bf x}^2\left|f\left({\bf t+\frac{x}{2} }\right)g^*\left({\bf t-\frac{x}{2} }\right)\right|^2d{\bf x}d{\bf t}\right)\left(\int_{\mathbb R^2}\int_{\mathbb R^2}{\bf w}^2|\mathbb W^{\Lambda_1,\Lambda_2,\Lambda_3}_{f,g}({\bf t,w})|^2d{\bf w}d{\bf t}\right)^\\
\nonumber&&\qquad\qquad\qquad\qquad\ge\frac{2}{\pi|b_3|}b^2_1b^2_2\|f\|^4_2\|g\|^4_2.\\
\end{eqnarray}

\end{theorem}

\begin{proof}For any signal $f\in L^1(\mathbb R^3,O)\cap L^2(\mathbb R^3,O),$ the Heisenberg’s uncertainty principle associated with octonion linear canonical transform reads \cite{li}
\begin{equation}\label{h1}
\int_{\mathbb R^2}{\bf x}^2|f({\bf x})|^2d{\bf x}\int_{\mathbb R^2}{\bf w}^2|\mathcal L^{\Lambda_1,\Lambda_2,\Lambda_3}_{\tau_1,\tau_2,\tau_4}[f]({\bf w})|^2d{\bf w}\ge\frac{2}{\pi|b_3|}b_1^2b_2^2\|f\|^2_2.
\end{equation}
Since $f,g\in L^1(\mathbb R^3,O)\cap L^2(\mathbb R^3,O),$ which implies $h_{\bf t}({\bf x})$ defined in (\ref{fun 3d h}) belongs to $L^1(\mathbb R^3,O)\cap L^2(\mathbb R^3,O)$. Therefore on replacing $f({\bf x})$ by $h_{\bf t}({\bf x})$ , (\ref{h1}) yields
\begin{equation}\label{h2}
\int_{\mathbb R^2}{\bf x}^2|h_{\bf t}({\bf x})|^2d{\bf x}\int_{\mathbb R^2}{\bf w}^2|\mathcal L^{\Lambda_1,\Lambda_2,\Lambda_3}_{\tau_1,\tau_2,\tau_4}[h_{\bf t}]({\bf w})|^2d{\bf w}\ge\frac{2}{\pi|b_3|}b_1^2b_2^2\|h_{\bf t}\|^2_2.
\end{equation}
Applying (\ref{3d WDOL in h}) to LHS of (\ref{h2}), we obtain
\begin{equation}\label{h3}
\int_{\mathbb R^2}{\bf x}^2|h_{\bf t}({\bf x})|^2d{\bf x}\int_{\mathbb R^2}{\bf w}^2|\mathbb W^{\Lambda_1,\Lambda_2,\Lambda_3}_{f,g}({\bf t,w})|^2d{\bf w}\ge\frac{2}{\pi|b_3|}b_1^2b_2^2\left(\int_{\mathbb R^2}|h_{\bf t}({\bf x})|^2d{\bf x}\right)^2.
\end{equation}
Then, we have 
\begin{eqnarray}\label{h4}
\nonumber&&\int_{\mathbb R^2}{\bf x}^2\left|f\left({\bf t+\frac{x}{2} }\right)g^*\left({\bf t-\frac{x}{2} }\right)\right|^2d{\bf x}\int_{\mathbb R^2}{\bf w}^2|\mathbb W^{\Lambda_1,\Lambda_2,\Lambda_3}_{f,g}({\bf t,w})|^2d{\bf w}\\
\nonumber&&\qquad\ge\frac{2}{\pi|b_3|}b_1^2b_2^2\left(\int_{\mathbb R^2}\left|f\left({\bf t+\frac{x}{2} }\right)g^*\left({\bf t-\frac{x}{2} }\right)\right|^2d{\bf x}\right)^2.\\
\end{eqnarray}
 To (\ref{h4}) we first take square root and then integrating it both sides with respect
to $d{\bf t}$, we get 
\begin{eqnarray}\label{h5}
\nonumber&&\int_{\mathbb R^2}\left\{\left(\int_{\mathbb R^2}{\bf x}^2\left|f\left({\bf t+\frac{x}{2} }\right)g^*\left({\bf t-\frac{x}{2} }\right)\right|^2d{\bf x}\right)^{1/2}\left(\int_{\mathbb R^2}{\bf w}^2|\mathbb W^{\Lambda_1,\Lambda_2,\Lambda_3}_{f,g}({\bf t,w})|^2d{\bf w}\right)^{1/2}\right\}d{\bf t}\\
\nonumber&&\qquad\ge\sqrt{\frac{2}{\pi|b_3|}}b_1b_2\int_{\mathbb R^2}\int_{\mathbb R^2}\left|f\left({\bf t+\frac{x}{2} }\right)g^*\left({\bf t-\frac{x}{2} }\right)\right|^2d{\bf x}d{\bf t}.\\
\end{eqnarray}
Applying the Cauchy–Schwarz inequality to the LHS
of (\ref{h5}), we have
\begin{eqnarray}\label{h6}
\nonumber&&\left(\int_{\mathbb R^2}\int_{\mathbb R^2}{\bf x}^2\left|f\left({\bf t+\frac{x}{2} }\right)g^*\left({\bf t-\frac{x}{2} }\right)\right|^2d{\bf x}d{\bf t}\right)^{1/2}\left(\int_{\mathbb R^2}\int_{\mathbb R^2}{\bf w}^2|\mathbb W^{\Lambda_1,\Lambda_2,\Lambda_3}_{f,g}({\bf t,w})|^2d{\bf w}d{\bf t}\right)^{1/2}\\
\nonumber&&\qquad\ge\sqrt{\frac{2}{\pi|b_3|}}b_1b_2\int_{\mathbb R^2}\int_{\mathbb R^2}\left|f\left({\bf t+\frac{x}{2} }\right)g^*\left({\bf t-\frac{x}{2} }\right)\right|^2d{\bf x}d{\bf t}.\\
\end{eqnarray}
Further applying the Fubini theorem and using suitable change of variables(just like earlier) to RHS of (\ref{h6}), we obtain
\begin{eqnarray}\label{h7}
\nonumber&&\left(\int_{\mathbb R^2}\int_{\mathbb R^2}{\bf x}^2\left|f\left({\bf t+\frac{x}{2} }\right)g^*\left({\bf t-\frac{x}{2} }\right)\right|^2d{\bf x}d{\bf t}\right)^{1/2}\left(\int_{\mathbb R^2}\int_{\mathbb R^2}{\bf w}^2|\mathbb W^{\Lambda_1,\Lambda_2,\Lambda_3}_{f,g}({\bf t,w})|^2d{\bf w}d{\bf t}\right)^{1/2}\\
\nonumber&&\qquad\ge\sqrt{\frac{2}{\pi|b_3|}}b_1b_2\|f\|^2_2\|g\|^2_2.\\
\end{eqnarray}
Which completes the proof.
\end{proof}
\begin{lemma}[Logarithmic Uncertainty Principle for the WD-QLCT]\label{log WD-QLCT}For $f,g\in S(\mathbb R^2,\mathbb H)$, we have the following inequality
\begin{eqnarray}\label{eqn log WD-QLCT}
\nonumber&&\int_{\mathbb R^2}\int_{\mathbb R^2}\ln|{\bf x}|\left| f\left({\bf t+\frac{x}{2} }\right) {g^*}\left({\bf t-\frac{x}{2} }\right)\right|^2d{\bf x}d{\bf t}+\int_{\mathbb R^2}\int_{\mathbb R^2}\ln|{\bf w}|\left|\mathbb W^{\Lambda_1,\Lambda_2}_{f, g} ({\bf t,w})\right|^2d{\bf w}d{\bf t}\\
&&\qquad\qquad\qquad\qquad\qquad\qquad\ge(D+\ln|{\bf b}|)| f|^2_2| g|^2_2,
\end{eqnarray}
where $\phi(\frac{1}{2})-\ln\pi,$ $\phi(t)=\frac{\Gamma'(t)}{\Gamma(t)}$ and $\Gamma$ is a Gamma function.
\end{lemma}
\begin{proof}
 Applying the procedure defined in proof  of Theorem 1 in \cite{Dnew} to the  Logarithmic Uncertainty Principle associated with QLCT \cite{mb}, we get the desired result.
\end{proof}

\begin{theorem}[Logarithmic Uncertainty Principle for the WDOL]Let $f,g\in \mathcal S(\mathbb R^3,\mathbb O),$then
\begin{eqnarray}
\nonumber&&\int_{\mathbb R^2}\int_{\mathbb R^2}\ln|{\bf x}|\left|f\left({\bf t+\frac{x}{2} }\right)g^*\left({\bf t-\frac{x}{2} }\right)\right|^2d{\bf x}d{\bf t}+{2\pi b_3}\int_{\mathbb R^2}\int_{\mathbb R^2}\ln|{\bf w}|\left|\mathbb W ^{\Lambda_1,\Lambda_2,\Lambda_3}_{f,g}({\bf t, w})\right|^2d{\bf w}d{\bf t}\\
\label{log}&&\qquad\qquad\qquad\qquad\qquad\qquad\ge(D+\ln|{\bf b}|)| f|^2_2|g|^2_2.
\end{eqnarray}
\end{theorem}

\begin{proof} From Lemma \ref{rel with WDQLCT},3D-WDOL has been divided into four WD-QLCTs as
\begin{eqnarray}\label{l1}
\nonumber\left|\mathbb W ^{\Lambda_1,\Lambda_2,\Lambda_3}_{f,g}({\bf t, w})\right|^2
\nonumber&=&\frac{1}{2\pi b_3}\left(\left|\mathbb W^{\Lambda_1,\Lambda_2}_{(\tilde f,\tilde g)^e} ({\bf t,w})\right|^2+\left|\mathbb W^{\Lambda_1,\Lambda_2}_{(\hat f,\hat g)^o} ({\bf t, w})\right|^2\right.\\
\nonumber&&+\left.\left|\mathbb W^{\Lambda_1,\Lambda_2}_{(\hat f,\hat g)^e} ({\bf t, w})\right|^2+\left|\mathbb W^{\Lambda_1,\Lambda_2}_{(\tilde f,\tilde g)^o} ({\bf t, w})\right|^2\right),\\
\end{eqnarray}
where $\tilde f,\, \hat f,\, \tilde g,\,\hat g\in L^2(\mathbb R^2,\mathbb H).$\\
Thus
\begin{eqnarray}\label{l2}
\nonumber&&{2\pi b_3}\int_{\mathbb R^2}\int_{\mathbb R^2}\ln|{\bf w}|\left|\mathbb W ^{\Lambda_1,\Lambda_2,\Lambda_3}_{f,g}({\bf t, w})\right|^2d{\bf w}d{\bf t}\\
\nonumber&&=\left[\int_{\mathbb R^2}\int_{\mathbb R^2}\ln|{\bf w}|\left|\mathbb W^{\Lambda_1,\Lambda_2}_{(\tilde f,\tilde g)^e} ({\bf t,w})\right|^2d{\bf w}d{\bf t}+\int_{\mathbb R^2}\int_{\mathbb R^2}\ln|{\bf w}|\left|\mathbb W^{\Lambda_1,\Lambda_2}_{(\hat f,\hat g)^o} ({\bf t, w})\right|^2d{\bf w}d{\bf t}\right.\\
\nonumber&&+\left.\int_{\mathbb R^2}\int_{\mathbb R^2}\ln|{\bf w}|\left|\mathbb W^{\Lambda_1,\Lambda_2}_{(\hat f,\hat g)^e} ({\bf t, w})\right|^2d{\bf w}d{\bf t}+\int_{\mathbb R^2}\int_{\mathbb R^2}\ln|{\bf w}|\left|\mathbb W^{\Lambda_1,\Lambda_2}_{(\tilde f,\tilde g)^o} ({\bf t, w})\right|^2d{\bf w}d{\bf t}\right].\\
\end{eqnarray}
Since every octonion function say $h_{\bf t}({\bf x})$ defined in (\ref{fun 3d h}) can be written in the quaternion form as:
\begin{equation}\label{l3}
|h_{\bf t}({\bf x})|^2=|\tilde h^e_{\bf t}({\bf x})|^2+|\tilde h^o_{\bf t}({\bf x})|^2+|\hat h^e_{\bf t}({\bf x})|^2+|\hat h^o_{\bf t}({\bf x})|^2.
\end{equation}
Therefore
\begin{eqnarray}\label{l4}
\nonumber\int_{\mathbb R^2}\int_{\mathbb R^2}\ln|{\bf x}||h_{\bf t}({\bf x})|^2d{\bf x}{\bf t}&=&\int_{\mathbb R^2}\int_{\mathbb R^2}\ln|{\bf x}||\tilde h^e_{\bf t}({\bf x})|^2d{\bf x}{\bf t}+\int_{\mathbb R^2}\int_{\mathbb R^2}\ln|{\bf x}||\tilde h^o_{\bf t}({\bf x})|^2d{\bf x}{\bf t}\\
\nonumber&&+\int_{\mathbb R^2}\int_{\mathbb R^2}\ln|{\bf x}||\hat h^e_{\bf t}({\bf x})|^2d{\bf x}{\bf t}+\int_{\mathbb R^2}\int_{\mathbb R^2}\ln|{\bf x}||\hat h^o_{\bf t}({\bf x})|^2d{\bf x}{\bf t}.\\
\end{eqnarray}
Now (\ref{l4}), implies
\begin{eqnarray}\label{l5}
\nonumber&&\int_{\mathbb R^2}\int_{\mathbb R^2}\ln|{\bf x}|\left|f\left({\bf t+\frac{x}{2} }\right)g^*\left({\bf t-\frac{x}{2} }\right)\right|^2d{\bf x}d{\bf t}\\
\nonumber&&\qquad\qquad=\int_{\mathbb R^2}\int_{\mathbb R^2}\ln|{\bf x}|\left|\left\{\tilde f\left({\bf t+\frac{x}{2} }\right)\tilde {g^*}\left({\bf t-\frac{x}{2} }\right)\right\}^e\right|^2d{\bf x}d{\bf t}\\
\nonumber&&\qquad\qquad+\int_{\mathbb R^2}\int_{\mathbb R^2}\ln|{\bf x}|\left|\left\{\tilde f\left({\bf t+\frac{x}{2} }\right)\tilde {g^*}\left({\bf t-\frac{x}{2} }\right)\right\}^o\right|^2d{\bf x}d{\bf t}\\
\nonumber&&\qquad\qquad+\int_{\mathbb R^2}\int_{\mathbb R^2}\ln|{\bf x}|\left|\left\{\hat f\left({\bf t+\frac{x}{2} }\right)\hat {g^*}\left({\bf t-\frac{x}{2} }\right)\right\}^e\right|^2d{\bf x}d{\bf t}\\
\nonumber&&\qquad\qquad+\int_{\mathbb R^2}\int_{\mathbb R^2}\ln|{\bf x}||\left|\left\{\hat f\left({\bf t+\frac{x}{2} }\right)\hat {g^*}\left({\bf t-\frac{x}{2} }\right)\right\}^o\right|^2d{\bf x}d{\bf t}.\\
\end{eqnarray}
By the Logarithmic Uncertainty Principle for the WD-QLCT given in (\ref{}), we get
\begin{eqnarray}\label{l6}
\nonumber&&\int_{\mathbb R^2}\int_{\mathbb R^2}\ln|{\bf w}|\left|\mathbb W^{\Lambda_1,\Lambda_2}_{(\tilde f,\tilde g)^e} ({\bf t,w})\right|^2d{\bf w}d{\bf t}+\int_{\mathbb R^2}\int_{\mathbb R^2}\ln|{\bf x}|\left|\left\{\tilde f\left({\bf t+\frac{x}{2} }\right)\tilde {g^*}\left({\bf t-\frac{x}{2} }\right)\right\}^e\right|^2d{\bf x}d{\bf t}\\
&&\qquad\qquad\qquad\qquad\qquad\qquad\ge(D+\ln|{\bf b}|)|\tilde f^e|^2_2|\tilde g^e|^2_2.
\end{eqnarray}
Similarly
\begin{eqnarray}\label{l7}
\nonumber&&\int_{\mathbb R^2}\int_{\mathbb R^2}\ln|{\bf w}|\left|\mathbb W^{\Lambda_1,\Lambda_2}_{(\tilde f,\tilde g)^o} ({\bf t,w})\right|^2d{\bf w}d{\bf t}+\int_{\mathbb R^2}\int_{\mathbb R^2}\ln|{\bf x}|\left|\left\{\tilde f\left({\bf t+\frac{x}{2} }\right)\tilde {g^*}\left({\bf t-\frac{x}{2} }\right)\right\}^o\right|^2d{\bf x}d{\bf t}\\
&&\qquad\qquad\qquad\qquad\qquad\qquad\ge(D+\ln|{\bf b}|)|\tilde f^o|^2_2|\tilde g^o|^2_2
\end{eqnarray}
And 
\begin{eqnarray}\label{l8}
\nonumber&&\int_{\mathbb R^2}\int_{\mathbb R^2}\ln|{\bf w}|\left|\mathbb W^{\Lambda_1,\Lambda_2}_{(\hat f,\hat g)^e} ({\bf t,w})\right|^2d{\bf w}d{\bf t}+\int_{\mathbb R^2}\int_{\mathbb R^2}\ln|{\bf x}|\left|\left\{\hat f\left({\bf t+\frac{x}{2} }\right)\hat {g^*}\left({\bf t-\frac{x}{2} }\right)\right\}^e\right|^2d{\bf x}d{\bf t}\\
&&\qquad\qquad\qquad\qquad\qquad\qquad\ge(D+\ln|{\bf b}|)|\hat f^e|^2_2|\hat g^e|^2_2.
\end{eqnarray}
And
\begin{eqnarray}\label{l9}
\nonumber&&\int_{\mathbb R^2}\int_{\mathbb R^2}\ln|{\bf w}|\left|\mathbb W^{\Lambda_1,\Lambda_2}_{(\hat f,\hat g)^o} ({\bf t,w})\right|^2d{\bf w}d{\bf t}+\int_{\mathbb R^2}\int_{\mathbb R^2}\ln|{\bf x}|\left|\left\{\hat f\left({\bf t+\frac{x}{2} }\right)\hat {g^*}\left({\bf t-\frac{x}{2} }\right)\right\}^o\right|^2d{\bf x}d{\bf t}\\
\nonumber&&\qquad\qquad\qquad\qquad\qquad\qquad\ge(D+\ln|{\bf b}|)|\hat f^o|^2_2|\hat g^o|^2_2.
\end{eqnarray}
Collecting equations (\ref{l6}), (\ref{l7}), (\ref{l8}) and(\ref{l9}), we obtain
\begin{eqnarray}
\nonumber&&{2\pi b_3}\int_{\mathbb R^2}\int_{\mathbb R^2}\ln|{\bf w}|\left|\mathbb W ^{\Lambda_1,\Lambda_2,\Lambda_3}_{f,g}({\bf t, w})\right|^2d{\bf w}d{\bf t}+\int_{\mathbb R^2}\int_{\mathbb R^2}\ln|{\bf x}|\left|f\left({\bf t+\frac{x}{2} }\right)g^*\left({\bf t-\frac{x}{2} }\right)\right|^2d{\bf x}d{\bf t}\\
\nonumber&&\qquad\qquad\qquad\qquad\qquad\qquad\ge(D+\ln|{\bf b}|)| f|^2_2|g|^2_2.\\
\end{eqnarray}
Which completes the proof.
\end{proof}
\begin{theorem}[Hausdorff-Young inequality for the WDOL]Let $2\le p<\infty$ and $f,g\in L^2(\mathbb R^2,\mathbb O),$ we have
\begin{equation}
\int_{\mathbb R^2}\int_{\mathbb R^2}\left|\mathbb W^{\Lambda_1,\Lambda_2,\Lambda_3}_{f,g}({\bf t,w})\right|^qd{\bf w}d{\bf t}\le E\frac{|b_1b_2|^{-\frac{1}{2}+\frac{1}{q}}}{(2\pi)^{\frac{1}{2q}+1}|b_3|^{\frac{1}{2q}}}\|f\|^q_2\|g\|^q_2.
\end{equation}
\end{theorem}
\begin{proof}For $1\le p\le2$ and $\frac{1}{p}+\frac{1}{q},$ the Hausdorff-Young inequality associated with  OLCT \cite{ownstolct} is given by
\begin{equation}\label{y1}
\left\|\mathcal L^{\Lambda_1,\Lambda_2,\Lambda_3}_{\tau_1,\tau_1,\tau_3}[f]\right\|_q\le\frac{|b_1b_2|^{-\frac{1}{2}+\frac{1}{q}}}{(2\pi)^{\frac{1}{2q}+1}|b_3|^{\frac{1}{2q}}}\|f\|_p.
\end{equation}
Making use of  (\ref{3d WDOL in h}), we can write
\begin{equation}\label{y2}
\left(\int_{\mathbb R^2}\left|\mathbb W^{\Lambda_1,\Lambda_2,\Lambda_3}_{f,g}({\bf t,w})\right|^qd{\bf w}\right)^{\frac{1}{q}}=\left(\int_{\mathbb R^2}\left|\mathcal L^{\Lambda_1,\Lambda_2,\Lambda_3}_{\tau_1,\tau_1,\tau_3}[h_{\bf t}]\right|^qd{\bf w}\right)^{\frac{1}{q}},
\end{equation}
where $$h_{\bf t}({\bf x})=f\left({\bf t+\frac{x}{2} }\right)g^*\left({\bf t-\frac{x}{2} }\right).$$
Therefore,on applying (\ref{y1}) to the right hand side of (\ref{y2}), we get
\begin{eqnarray}\label{y3}
\nonumber\left(\int_{\mathbb R^2}\left|\mathbb W^{\Lambda_1,\Lambda_2,\Lambda_3}_{f,g}({\bf t,w})\right|^qd{\bf w}\right)^{\frac{1}{q}}&\le&\frac{|b_1b_2|^{-\frac{1}{2}+\frac{1}{q}}}{(2\pi)^{\frac{1}{2q}+1}|b_3|^{\frac{1}{2q}}}\|h_{\bf t}\|_p\\
\nonumber&=&\frac{|b_1b_2|^{-\frac{1}{2}+\frac{1}{q}}}{(2\pi)^{\frac{1}{2q}+1}|b_3|^{\frac{1}{2q}}}\left\|f\left({\bf t+\frac{x}{2} }\right)g^*\left({\bf t-\frac{x}{2} }\right)\right\|_p\\
\nonumber&=&\frac{|b_1b_2|^{-\frac{1}{2}+\frac{1}{q}}}{(2\pi)^{\frac{1}{2q}+1}|b_3|^{\frac{1}{2q}}}\left(\int_{\mathbb R^2}\left |f\left({\bf t+\frac{x}{2} }\right)g^*\left({\bf t-\frac{x}{2} }\right)\right|^pd{\bf x}\right)^{\frac{1}{p}}.\\
\end{eqnarray}
Further simplifying, we have
\begin{eqnarray}\label{y4}
\nonumber\left(\int_{\mathbb R^2}\left|\mathbb W^{\Lambda_1,\Lambda_2,\Lambda_3}_{f,g}({\bf t,w})\right|^qd{\bf w}\right)&\le&\frac{|b_1b_2|^{-\frac{1}{2}+\frac{1}{q}}}{(2\pi)^{\frac{1}{2q}+1}|b_3|^{\frac{1}{2q}}}\left(\int_{\mathbb R^2}\left |f\left({\bf t+\frac{x}{2} }\right)g^*\left({\bf t-\frac{x}{2} }\right)\right|^pd{\bf x}\right)^{\frac{q}{p}}.\\
\end{eqnarray}
Integrating (\ref{y4}) both sides with respect $d{\bf t},$ it yields
\begin{eqnarray}\label{y5}
\nonumber&&\int_{\mathbb R^2}\left(\int_{\mathbb R^2}\left|\mathbb W^{\Lambda_1,\Lambda_2,\Lambda_3}_{f,g}({\bf t,w})\right|^qd{\bf w}\right)d{\bf t}\\
\nonumber&&\qquad\qquad\qquad\le\frac{|b_1b_2|^{-\frac{1}{2}+\frac{1}{q}}}{(2\pi)^{\frac{1}{2q}+1}|b_3|^{\frac{1}{2q}}}\int_{\mathbb R^2}\left(\int_{\mathbb R^2}\left |f\left({\bf t+\frac{x}{2} }\right)g^*\left({\bf t-\frac{x}{2} }\right)\right|^pd{\bf x}\right)^{\frac{q}{p}}d{\bf t}.\\
\end{eqnarray}
Using procedure defined in \cite{20aa}( see  theorem 1  relation (3.3)), we have
\begin{equation}\label{y6}
\int_{\mathbb R^2}\left(\int_{\mathbb R^2}\left |f\left({\bf t+\frac{x}{2} }\right)g^*\left({\bf t-\frac{x}{2} }\right)\right|^pd{\bf x}\right)^{\frac{q}{p}}d{\bf t}\le E\left\{\|f\|_2\|g\|_2\right\}^q,
\end{equation}
where E is positive constant.\\
Thus (\ref{y5}) and (\ref{y6}) together yields
\begin{equation*}
\int_{\mathbb R^2}\int_{\mathbb R^2}\left|\mathbb W^{\Lambda_1,\Lambda_2,\Lambda_3}_{f,g}({\bf t,w})\right|^qd{\bf w}d{\bf t}\le E\frac{|b_1b_2|^{-\frac{1}{2}+\frac{1}{q}}}{(2\pi)^{\frac{1}{2q}+1}|b_3|^{\frac{1}{2q}}}\|f\|^q_2\|g\|^q_2.
\end{equation*}
Which completes the proof.
\end{proof}

\section{Conclusion}\label{sec 5}

In the present study, we examined  three major objectives: first, we have studied  the notion of
 1D-WDOL in the framework of time-frequency analysis and examined  all
of its basic properties by means of the one dimensional octonion linear canonical transform transforms. Second, we propose the definition of 3D-WDOL and establish the fundamental properties associated with it, which 
includes the  reconstruction formula, Rayleigh's theorem and Riemann-Lebesgue Lemma.
Moreover, we also establish the relation of WDOL with 3D-LCT and WD-QLCT. Third, well known uncertainty  principles (UPs) including Heisenberg's UP, Logarithmic UP and Hausdorff-Young inequality associated with WDOL are established.
\section*{Declarations}
\begin{itemize}
\item  Availability of data and materials: The data is provided on the request to the authors.
\item Competing interests: The authors have no competing interests.
\item Funding: No funding was received for this work
\item Author's contribution: Both the authors equally contributed towards this work.
\item Acknowledgements: This work is supported by the  Research  Grant\\
(No. JKST\&IC/SRE/J/357-60) provided by JKSTIC, Govt. of Jammu and Kashmir,
India.

\end{itemize}


\begin{thebibliography}{99}
\bibitem{stoft} Gao,W.B and Li, B.Z: Octonion short-time Fourier transform for
time-frequency representation and its applications. IEEE Transactions on signal processing.
\end{thebibliography}
\end{document}